\documentclass[11pt,a4paper]{article}
\pdfoutput=1
\usepackage[backend=bibtex,bibencoding=utf8,citestyle=alphabetic,bibstyle=alphabetic,isbn=false,doi=false,url=false,giveninits=true,hyperref=true,backref=true,backrefstyle=three+]{biblatex}
\renewbibmacro*{volume+number+eid}{%
	\textbf{\printfield{volume}}%
	\setunit*{\addnbspace}
	\printfield{number}%
	\setunit{\addcomma\space}%
	\printfield{eid}
}
\renewbibmacro{in:}{%
	\ifentrytype{article}{}{\printtext{\bibstring{in}\intitlepunct}}}
\DeclareFieldFormat[article]{number}{\mkbibparens{#1}}
\usepackage[UKenglish]{babel}
\usepackage{amsmath,amssymb,amsthm}
\usepackage{mathrsfs}
\usepackage{verbatim}
\usepackage[utf8]{inputenc}
\usepackage[T1]{fontenc}
\usepackage{lmodern}
\usepackage{graphicx}
\usepackage{wrapfig}
\usepackage{microtype}
\usepackage[normalem]{ulem} 
\usepackage[retainorgcmds]{IEEEtrantools}
\usepackage[showframe=false]{geometry}
\usepackage[colorlinks]{hyperref}
\usepackage[capitalise
]{cleveref}
\usepackage{tikz-cd}
\usepackage{csquotes}
\usepackage{mathtools}
\allowdisplaybreaks
\crefformat{equation}{(#2#1#3)}
\crefrangeformat{equation}{(#3#1#4) to~(#5#2#6)}
\crefmultiformat{equation}{(#2#1#3)}%
{ and~(#2#1#3)}{, (#2#1#3)}{ and~(#2#1#3)}

\newcommand\ie{{\em i.e.}~}

\def\N{\mathbb{N}}
\def\T{\mathbb{T}}
\def\Q{\mathbb{Q}}
\def\Z{\mathbb{Z}}
\def\C{\mathbb{C}}

\def\U{\mathscr U}
\def\F{\mathscr F}
\def\I{\mathscr I}
\def\J{\mathscr J}
\def\E{\mathscr E}
\def\A{\mathcal A} 
\def\D{{\mathcal D}}
\def\f{\mathbf f}
\def\H{\mathfrak H}
\def\HH{\mathcal H}
\def\K{\mathfrak K}
\def\M{{\mathcal M}}

\def\P{\mathcal P}
\def\R{\mathbb{R}}

\def\V{\mathcal V}
\def\X{\mathcal X}
\def\Y{\mathcal Y}

\def\S{\mathcal S}

\def\FF{{\mathfrak F}}
\def\W{{\mathcal W}}

\def\la{\langle}
\def\ra{\rangle}

\def\Id{\mathbb I}
\def\d{\mathrm d}
\def\z{\mathsf{z}}
\def\B{\mathbb{B}}

\def\({\left(}
\def\[{\left[}
\def\){\right)}
\def\]{\right]}
\def\l|{\left\lvert}
\def\r|{\right\rvert}
\def\lp{\left\lVert}
\def\rp{\right\rVert}

\def\G{\mathcal{G}}

\def\p{\parallel}
\def\<{\langle}
\def\>{\rangle}

\def\Op{\mathfrak{Op}}

\def\a{\mathbf a}

\def\XX{\mathfrak{X}}
\def\xx{\mathfrak{x}}
\def\ff{\mathfrak{f}}
\def\ee{\mathfrak{e}}
\def\e{\mathrm e}
\def\VV{\tilde{V}}

\def\J{\mathscr{J}}

\DeclareMathOperator{\Trigpol}{Trig\,Pol}

\newtheorem{Theorem}{Theorem}[section]
\newtheorem{Remark}[Theorem]{Remark}
\newtheorem{Lemma}[Theorem]{Lemma}

\newtheorem{Proposition}[Theorem]{Proposition}
\newtheorem{Definition}[Theorem]{Definition}

\def\n{\mathbf{n}}
\def\s{\mathbf{s}}
\def\D{\mathbf{D}}
\def\Q{\mathbf{Q}}
\def\mg{m_\Gamma}
\DeclareMathOperator{\Ran}{Ran}
\begin{document}

\title{Spectral and scattering theory for Gauss--Bonnet operators on perturbed topological crystals}

\author{D. Parra$^1$\footnote{This work was partially supported by the LABEX MILYON (ANR-10-LABX-0070) of Université de Lyon, within the program "Investissements d'Avenir" (ANR-11-IDEX-0007) operated by the French National Research Agency (ANR).}}

\date{\small}
\maketitle \vspace{-1cm}

\begin{quote}
\emph{
\begin{itemize}
\item[$^1$] Univ Lyon, Université Claude Bernard Lyon 1, CNRS UMR 5208, Institut Camille Jordan, 43 blvd. du 11 novembre 1918, F-69622 Villeurbanne cedex, France
\item[] \emph{E-mail:} parra@math.univ-lyon1.fr
\end{itemize}
}
\end{quote}

\begin{abstract}
In this paper we investigate the spectral and the scattering theory of Gauss--Bonnet operators acting on perturbed periodic combinatorial graphs. Two types of perturbation are considered: either a multiplication operator by a short-range or a long-range potential, or a short-range type modification of the graph. For short-range perturbations, existence and completeness of local wave operators are proved. In addition, similar results are provided for the Laplacian acting on edges.
\end{abstract}


\smallskip


\section{Introduction}
In this paper we continue our study of discrete differential operators on perturbed topological crystals initiated in \cite{PR16x}. Here we study the Gauss--Bonnet operator $D=d+d^*$ as recently introduced in \cite{AT15}, see also \cite{GH14}. This is a Dirac-type operator acting both on vertices and edges. Since this operator has received much less attention compared with the combinatorial Laplacian acting on vertices, only few results are known. In \cite{AT15} the problem of essential self-adjointness for locally finite graphs is investigated. This is connected with the attention that has recently received this question for the Laplacian acting on vertices $\Delta_0$, see \cite{GS11,Ke15,MT14} and references therein.

Topological crystals are a class of periodic graphs that arise as natural generalization of the graphs used to model crystals. We refer to the monograph \cite{Su12} for a general treatment of such graphs. The spectral theory for the Gauss--Bonnet operator has not been fully developed. We refer although to \cite{GH14} where for the simpler case $\Z$ the preservation of the absolutely continuous spectrum under multiplicative perturbations was proved under some integrability assumptions; however their conditions are not directly related to the conditions that we give below. More importantly, we also consider perturbations of the periodic graph itself. This perturbation is encoded by a non periodic measure that converges at infinity to a periodic measure. This kind of metric perturbation has received some attention lately for the Laplacian acting on vertices, but only considering compactly supported perturbations \cite{AIM15} or restricting the study to the essential spectrum \cite{SS15x}. 

Our main result is that the spectrum of $D$ has a band structure that mimics the band structure of $\Delta_0$. Furthermore, we show that under short range metric perturbations and short and long range multiplicative perturbations this structure is preserved whith only finitely many eigenvalues added to it, multiplicities taken into account. The main tools used for obtaining such results are Floquet-Bloch decomposition and Mourre Theory. In fact we will extensively use abstract results coming from \cite{PR16x} that for convenience we sum up in \cref{abstracto}. For a general presentation of Mourre Theory we refer to \cite{ABG96} and to \cite{GN98} for the method that stands at the basis of the approach used in \cite{PR16x}.

In the last section of the paper, we also study the spectral and the scattering theory for the Laplacian acting on edges $\Delta_1$. This operator is dual to the Laplacian acting on vertices for which the spectral and the scattering theory were developed in \cite{PR16x}. It has received less attention than $\Delta_0$ but we can still refer to \cite{BGJ15x} where the problem of essential self-adjointness is addressed.

The paper is organized as follows. In \cref{sec:main} we define the operator and the type of periodic graphs that we consider and give our main statement in \cref{principal}. In \cref{sec:decom} we decompose the Gauss--Bonnet operator into a fibered operator. We provide the poof of our main theorem in \cref{sec:proof} and in \cref{sec:laplacian} we show a similar result for the Laplacian acting on edges.
\section{Preliminaries and main result}\label{sec:main}
In this section we define the Gauss--Bonnet operator, the type of graphs that we consider and we state our main result.

\subsection{The Gauss--Bonnet operator}

A graph $X=\big(V(X),E(X)\big)$ is composed of a set of vertices $V(X)$ and a set of unoriented edges $E(X)$. Generically, we shall use the notation $x,y$ for elements of $V(X)$, and $\e=\{x,y\}$ for elements of $E(X)$. Graphs can have loops and multiple edges. If both $V(X)$ and $E(X)$ are finite sets, the graph $X$ is said to be finite.  From the set of unoriented edges $E(X)$ we construct $ A(X)$, the set of oriented edges, considering for every unoriented edge $\{x,y\}$ both oriented edges $(x,y)$ and $(y,x)$ in $ A(X)$. The origin vertex of an oriented edge $\e$ is denoted by $o(\e)$, the terminal one by $t(\e)$, and $\overline{\e}$ denotes the edge obtained from $\e$ by interchanging the vertices, \ie $o(\overline{\e})=t(\e)$ and $t(\overline{\e})=o(\e)$. For a vertex $x\in V(X)$ we also set $ A_x=\{\e\in A(X)\mid o(\e)=x\}$. If $A_x$ is finite for every $x\in V(X)$ we say that $X$ is locally finite.

We consider the vector spaces of \emph{$0-$cochains} $C^0(X)$ and \emph{$1-$cochains} $C^1(X)$ given by:
\begin{equation*}
C^0(X):=\{f:V(X)\to\C\}\text{ ; }\qquad C^1(X):=\{f: A(X)\to\C\mid f(\e)=-f(\overline{\e})\}\text{ .}
\end{equation*}
We will denote by $C(X)$ the direct sum $C^{0}(X)\oplus C^{1}(X)$ and we will refer to an $f\in C{(X)}$ as a \emph{cochain}. This suggests the notation $C_c(X)$ for the space of cochains that are finitely supported.
 We will also use this notation for the subspaces of $0$-cochains and $1$-cochains, \ie $C^{j}_c(X)=C^{j}(X)\cap C_c(X)$.

We say that $X$ is a weighted graph if it is equipped with a positive measure $m$ defined on both vertices and non-oriented edges. We will only consider measures with full support. On oriented edges, a measure satisfies $m(\e)=m(\overline{\e})$. Given a weighted graph $(X,m)$ we define
\begin{equation*}
\deg_m:V(X)\to(0,\infty]\text{ ; }\deg_m (x)=\sum_{\e\in A_x}\frac{m(\e)}{m(x)}
\end{equation*}
as the natural generalization of the combinatorial degree defined by $\sharp  A_x$. For $f,g\in C_c(X)$ we can define an inner product by
\begin{equation}\label{interno}
\la f,g\ra:=\sum_{x\in V(X)}m(x)f(x)\overline{g(x)}+\frac12 \sum_{\e\in A(X)}m(\e)f(\e)\overline{g(\e)}\text{.}
\end{equation}
The Hilbert space $l^2(X,m)$ is defined as the closure of $C_c(X)$ in the norm induced by \cref{interno}. It coincides with $\{f\in C(X)\mid\lp f\rp:=\la f,f\ra^{\frac12}<\infty\}$. By taking the restriction on each component of $C(X)$, one can also define the Hilbert spaces $l^2_0(X,m)$ and $l^2_1(X,m)$. When needed we will denote by $\lp\cdot\rp_0$ and $\lp\cdot\rp_1$ the corresponding norms and by $\la\cdot,\cdot\ra_0$ and $\la\cdot,\cdot\ra_1$ the corresponding inner products.

We can now define the \emph{coboundary operator} $d:C^0_c(X)\to C^1(X)$ by:
\begin{equation*}
df(\e):=f(t(\e))-f(o(\e))\text{.}
\end{equation*}
Its formal adjoint $d^*:C^1_c(X,m)\to C^0_c(X,m)$ is given by
\begin{equation*}
d^*f(x)=-\sum_{\e\in A_x}\frac{m(\e)}{m(x)}f(\e)
\end{equation*}
and corresponds to the \emph{boundary operator}. We can extend both operators by zero to get $d,d^*:C_c(X)\to C(X)$. Then we can define the Gauss--Bonnet operator $D(X,m)$ by
\begin{equation*}
D\equiv D(X,m):C_c(X)\to C(X)\quad\text{ ; }\quad D(X,m):=d+d^*\ .
\end{equation*}
It is a symmetric operator by construction. If $\deg_m$ is bounded, it can be shown that $D$ extends to a bounded operator on $l^2(X,m)$. Henceforth we assume the boundedness of $\deg_m$ since it will be the case for the periodic graphs that we shall consider. Furthermore, since $d^2=0=(d^{*})^2$, $D$ satisfies
\begin{equation}\label{deltasquare}
D(X,m)^2=dd^*+d^*d=-\Delta_0(X,m)-\Delta_1(X,m)
\end{equation}
where $\Delta_0(X,m)$ is the Laplacian of a graph acting on vertices and $\Delta_1(X,m)$ is the Laplacian acting on edges. It follows that $D$ should be considered like an analog of Dirac-type operators on manifolds because its square is a Laplacian-type operator. In fact, $-D^2$ is the discrete analog of the Hodge Laplacian, a particular case of \emph{higher order combinatorial Laplacians}, see for instance \cite{Ec45,HJ13,SKM14}. 

\begin{Remark}
When one wants to stress the fact that $l^2(X,m)=l^2_0(X,m)\oplus l^2_1(X,m)$, $D(X,m)$ is usually written in matrix form. Then \cref{deltasquare} is rewritten as:
\begin{equation*}
-\Delta(X,m):=D(X,m)^2=\begin{pmatrix}
0& d^{*}\\
d& 0
\end{pmatrix}^2=
\begin{pmatrix}
-\Delta_0(X,m)& 0\\
0& -\Delta_1(X,m)
\end{pmatrix}\text{ ,}
\end{equation*}
where $\Delta(X,m)$ stands for the Hodge Laplacian of the weighted graph $(X,m)$.
\end{Remark}

\subsection{Topological crystals}
We turn now to the particular kind of periodic graphs that we consider, known as \emph{topological crystal}. We follow the presentation of \cite{Su12} and we take this opportunity to settle some notations that will be useful for the statement of our main theorem.

A morphism $\omega$ between two graphs $X$ and $\XX$ is composed of a pair of maps $\omega:V(X)\to V(\XX)$ and $\omega:A(X)\to A(\XX)$ such that it preserves the adjacency relations between vertices and edges. Namely, for $\e\in A(X)$ we have that $\omega(\e)=(\omega(o(\e)),\omega(t(\e)))$. An isomorphism is a morphism that is a bijection on vertices and edges. We denote by $\operatorname{Aut}(X)$ the group of isomorphisms of a graph $X$ into itself. A morphism $\omega:X\to\XX$ is said to be a covering map if
\begin{enumerate}
	\item $\omega:V(X)\to V(\XX)$ is surjective,
	\item for all $x\in V(X)$, $\omega|_{A_x}:A_x\to A_{\omega(x)}$ is a bijection.
\end{enumerate}
In such a case we say that $X$ is a \emph{covering graph} over the \emph{base graph} $\XX$. The \emph{transformation group} of a covering is defined as the subgroup $\Gamma\equiv\Gamma(\omega)$ of $\operatorname{Aut}(X)$ such that for every $\mu\in\Gamma(\omega)$ the equality $\omega\circ\mu=\omega$ holds. We define a topological crystal \cite[Sec. 6.2]{Su12} as follows.
\begin{Definition}\label{topocrystal}
	A $d$-dimensional topological crystal is a quadruplet $(X,\XX,\omega,\Gamma)$ such that:
	\begin{enumerate}
		\item $X$ is an infinite graph,
		\item $\XX$ is a finite graph,\label{xxfinite}
		\item $\omega:X\to \XX$ is a covering,
		\item $\Gamma$ is the transformation group of $\omega$ and is isomorphic to $\Z^d$,\label{gammaz}
		\item $\omega$ is regular, \ie for every $x$, $y\in V(X)$ satisfying $\omega(x)=\omega(y)$ there exists $\mu\in\Gamma$ such that $x=\mu y$.\label{omegaregular}
	\end{enumerate}
\end{Definition}
We usually say that $X$ is a topological crystal meaning that it admits a topological crystal structure given by $(X,\XX,\omega,\Gamma)$. Note that by \cref{xxfinite} and the definition of a covering all topological crystals are locally finite. To simplify the notation we will denote by $x$ (resp. $\xx$) the vertices, and by $\e$ (resp. $\ee$) the edges in $X$ (resp. in $\XX$).

It follows from \cref{omegaregular} in \cref{topocrystal} that $X/\Gamma\cong \XX$ and hence we can identify $V(\XX)$ with a subset of $V(X)$. Namely, since by assumption $V(\XX)=\{\xx_1,\dots,\xx_n\}$ for some $n\in\N$, we can choose $\VV=\{x_1,\dots,x_n\}\subset V(X)$ such that $\omega(x_i)=\xx_i$. For simplicity we will also use the notation $\check{x}:=\omega(x)$ and, once $\VV$ has been chosen, we write $\hat{\xx}$ for the vertex $x\in\VV$ such that $\check{x}=\xx$. We index vertices by $\imath :V(\XX)\to \{1,\dots,n\}$ such that $\xx=\xx_{\imath(\xx)}$. 

On the other hand, we can also identify $A(\XX)$ with a subset of $A(X)$ by setting $\tilde{A}=\bigcup_{x\in\tilde{V}} A_{x}\subset A(X)$ and hence similar notations will hold for edges. Namely, we have $\check{\e}=\omega(\e)$ and $\hat{\ee}\in\tilde{ A}$ such that $\omega(\hat{\ee})=\ee$. Furthermore, if in $ A(\XX)$ we consider an ordering $ A(\XX)=\{\ee_1,\overline{\ee}_1,\dots,\ee_l,\overline{\ee}_l\}$ we can define $\imath: A(\XX)\to \{1,\dots,l\}$ by $\ee=\ee_{\imath(\ee)}$ or $\ee=\overline{\ee}_{\imath(\ee)}$. This ordering in $ A(\XX)$ induces a decomposition of $ A(X)$ in two parts $ A(X)= A^{+}(X)\cup A^{-}(X)$ by $ A^{+}(X)=\{\e\in A(x)\mid\check{\e}=\ee_{\imath(\e)}\}$ and $ A^{-}(X)=\{\e\in A(x):\check{\e}=\overline{\ee}_{\imath(\e)}\}$. This decomposition is usually referred to as choosing an orientation on $X$. We will also use the decomposition $A(\XX)=A^{+}(\XX)\cup A^{-}(\XX)$. 

Note that all these notations depend only on the choice of $\VV$. With this choice we can also define the entire part of a vertex $x$, as the map $\[\,\cdot\,\]:V(X)\to\Gamma$ satisfying $\[x\]\widehat{\check{x}}=x$, and the entire part of an edge, as the map $\[\,\cdot\,\]: A(X)\to\Gamma$ satisfying $\[\e\]\widehat{\check{\e}}=\e$. The convention on $\tilde{ A}$ implies $\[\e\]=\[o(\e)\]$. We define the map 
\begin{equation}\label{etadef}
\eta: A(X)\to\Gamma\quad \text{;}\quad \eta(\e)=\[t(\e)\]\[o(\e)\]^{-1}
\end{equation}
 and we call $\eta(\e)$ the \emph{index} of the edge $\e$. For any $\mu\in\Gamma$ we have 
\begin{equation*}
\eta(\mu\e)=\[t(\mu\e)\]\[o(\mu\e)\]^{-1}=\mu\[t(\e)\]\mu^{-1}\[o(\e)\]^{-1}=\eta(\e)\text{ .}
\end{equation*}
This periodicity enables us to define $\eta(\ee)=\eta(\hat{\ee})$ unambiguously for every $\ee\in A (\XX)$. Again, this index in $ A(\XX)$ depends only on the choice of $\VV$.

\subsection{Statement of the main theorem}
A weighted topological crystal is  a quintuplet $(X,\XX,\omega,\Gamma,\mg )$ such that $(X,\XX,\omega,\Gamma)$ is a topological crystal, $(X,\mg )$ is a weighted graph and the structures are compatible in the sense that $\mg $ is $\Gamma$-periodic, \ie
\begin{equation}\label{mperiodica}
\mg (x)=\mg (\mu x) \text{ and } \mg (\e)=\mg (\mu\e) \text{ for every } x\in V(X), \e\in E(X) \text{ and } \mu\in\Gamma \text{ .} 
\end{equation}

We denote by $R:X\to\R$ a real valued function defined on both vertex and unoriented edges. We still denote by $R$ the multiplication operator defined for $f\in l^2(X,\mg )$ by
\begin{equation*}
\[Rf\](x)=R(x)f(x)\text{ ; }\[Rf\](\e)=R(\e)f(\e)\quad \forall x\in V(X)\text{ and }\e\in A(X)\ .
\end{equation*} 
We call such $R$ a potential. Note that a potential satisfies $R(\e)=R(\overline{\e})$. 

We will denote by $R_\Gamma$ a periodic potential in the sense of \cref{mperiodica}. Then we can define $H_0$ on $l^{2}(X,\mg )$ by
\begin{equation}\label{h0def}
H_0=D(X,\mg )+R_{\Gamma}\text{ .}
\end{equation}

It is known that the spectrum of $-\Delta_0(X,\mg )+R_\Gamma$ on $l^2_0(X,\mg)$ consists in a finite number of intervals of absolutely continuous spectrum and a finite number of eigenvalues of infinite multiplicity (\cite{BS99,HN09,KS14}). Our purpose here is to show that this holds true for $H_0$ and to study the essential stability of this spectral structure under perturbations. We will consider two types of perturbations. On the one hand we consider a perturbation of the measure $\mg $ by endowing $X$ with another measure $m$. In order to study this measure as a perturbation of the periodic measure $\mg $ we set $\J:l^{2}(X,m)\to l^{2}(X,\mg )$ by:
\begin{equation*}
\J f(x)=\(\frac{m(x)}{\mg (x)}\)^{\frac12}f(x)\quad\text{;}\quad\J f(\e)=\(\frac{m(\e)}{\mg (\e)}\)^{\frac12}f(\e)\text{.}
\end{equation*}
Since both measures have full support $\J$ is unitary. On the other hand, we consider also a pertubation of the periodic potential. This perturbation is defined by a potential $R$ that converges rapidly enough to $R_\Gamma$ at infinity. Then we can define our \emph{perturbed operator} $H$ on $l^2(X,\mg)$ by
\begin{equation}\label{full}
H:=\J (D(X,m)+R)\J^{*}=\J D(X,m)\J^{*}+R\ .
\end{equation}
Note that the multiplication operators commute with $\J$. The main result is described in the following theorem where we use the isomorphism between $\Gamma$ and $\Z^{d}$ to induce in the former a norm denoted by $|\cdot|$. 

\begin{Theorem}\label{principal}
	Let $X$ be a topological crystal. Let $H_0$ and $H$ be defined by \cref{h0def} and \cref{full} respectively. Assume that $m$ satisfies
	\begin{equation}
	\int_{1}^{\infty}\,\d \lambda\!\sup_{\lambda<|\[\e\]|<2\lambda}\left|\frac{m(\e)}{m(o(\e))}-\frac{\mg (\e)}{\mg (o(\e))}\right|<\infty\text{. }\label{mshortD}
	\end{equation}
	Assume also that the difference $R-R_\Gamma$ is equal to $R_S+R_L$ which satisfy
	\begin{equation}\label{potshort}
	\int_{1}^{\infty}\,\d \lambda\!\sup_{\lambda<|\[\e\]|<2\lambda}\max\{|R_S(\e)|,|R_S(o(\e))|\}<\infty\text{ ,}
	\end{equation}
	and
	\begin{equation}
	R_L\xrightarrow{x,\e\to\infty}0, \quad \hbox{ and } \quad
	\int_{1}^{\infty}\d\lambda \sup_{\lambda<|\[\e\]|<2\lambda}\big|R_L\big(\e\big)-R_L\big(o(\e)\big)\big|<\infty\ .\label{longGauss}
	\end{equation}
	Then there exists a discrete set $\tau\subset\R$ such that for every closed interval $I\subset\R\backslash\tau$ the following assertions hold in $I$:
	\begin{enumerate}
		\item $H_0$ has purely absolutely continuous spectrum,
		\item $H$ has no singular continuous spectrum and has at most a finite number of eigenvalues, each of finite multiplicity,
		\item if $R_L\equiv 0$ the local wave operators \begin{equation*}
		W_{\pm}\equiv W_{\pm}(H,H_0;I)=s-\lim_{t\to\pm\infty}e^{iH t}e^{-iH_0 t}E_{H_0}(I)
		\end{equation*}
		exist and are complete, \ie $\Ran(W_-)=\Ran(W_+)=E^{ac}_H(I)l^2(X,\mg)$. Since $\delta_{sc}(H)=\emptyset$ they are indeed asymptotically complete.
	\end{enumerate}
\end{Theorem}
Both \cref{mshortD} and \cref{potshort} describe a \emph{short-range} type of decay at infinity while \cref{longGauss} is usually referred as a \emph{long-range} type of decay.

\section{Integral decomposition}\label{sec:decom}
The aim of this section is to show that $H_0$ can be decomposed into a direct integral. For $\Delta_0(X,\mg)$ this decomposition has been an important tool for studying its spectral properties (see \cite{An13,HN09,KS14,KSS98}). The integral decomposition of the Gauss--Bonnet operator seems not to have been observed before but could have been easily deduced from the arguments presented in \cite[Sec. 6]{KS15x}. We first define magnetic operators, then we show that $H_0$ can be decomposed into the direct integral of magnetic Gauss--Bonnet operators defined on the small graph $\XX$ and finally we show how this decomposition can be transformed into a constant fiber decomposition.

\subsection{Discrete magnetics operators on a graph}
Let us start by introducing the magnetic scheme for a general weighted graph $(X,m)$ with bounded degree. The discrete analogue for a magnetic Laplacian operator has been widely studied \cite{CTT11a,HS01,HS99a,KS15x,Su94}. Let $\T$ denote the multiplicative group of complex numbers of modulus 1. For any $\theta:A(X)\to\T$ satisfying $\theta(\overline{\e})=\overline{\theta(\e)}$ one defines the \emph{magnetic Laplacian acting on vertices} for suitable $f$ by
\begin{equation*}
\[\Delta_{0}(X_\theta,m)f\](x)=\sum_{\e\in A_x}\frac{m(\e)}{m(x)}\Bigl(\theta(\e)f(t(\e))-f(x)\Bigr)\text{.}
\end{equation*}
There are several ways to adapt the coboundary operator $d$ to fit this magnetic scheme. We choose to set the space of \emph{magnetics 1-cochains} by
\begin{equation*}
C^1(X_\theta):=\{f:A(X)\to\C\mid f(\overline{\e})=-\overline{\theta(\e)}f(\e)\}\ .
\end{equation*}
We denote by $l^2(X_\theta,m)$ the Hilbert space defined as the closure of $C_c(X_\theta,m)=C^0_c(X,m)\oplus C^1_c(X_\theta,m)$ in the norm still induced by the inner product defined by \cref{interno}. Note that the space of $0$-cochains remains unchanged by $\theta$. Setting $d_\theta:C^{0}_c(X)\to C^{1}(X_\theta)$ by:
\begin{equation}\label{dmagnetic}
d_\theta f(\e)=\theta(\e)f(t(\e))-f(o(\e))\text{ ,}
\end{equation}
we get the desired factorization of $\Delta_{0}(X_\theta,m)$ as shown in the next lemma.
\begin{Lemma}
Let $d_\theta$ be defined by \cref{dmagnetic}. It is a well defined bounded operator, with adjoint given by 
\begin{equation*}
d^*_\theta f(x)=-\sum_{\e\in A_x}\frac{m(\e)}{m(x)}f(\e)\text{ .}
\end{equation*}
It satisfies $-d^*_\theta d_\theta=\Delta_{0}(X_\theta,m)$. The \emph{magnetic Laplacian acting on edges} is given by
\begin{equation*}
\[\Delta_{1}(X_\theta,m)f\](\e)=\sum_{\e'\in A_{t(\e)}}\frac{m(\e')}{m(t(\e))}\theta(\e) f(\e')-\sum_{\e'\in A_{o(\e)}}\frac{m(\e')}{m(o(\e))} f(\e')\text{.}
\end{equation*}
\end{Lemma}
\begin{proof}
	It is easy to see that 
	\begin{equation*}
	d_\theta f(\overline{\e})=\theta(\overline{\e})f(t(\overline{\e}))-f(o(\overline{\e}))=-\overline{\theta(\e)}(-f(o(\e))+\theta(\e)f(t(\e))=-\overline{\theta(\e)}d_\theta f(\e)
	\end{equation*}
	to get that $d_\theta f\in C^1(X_\theta)$. Furthermore, for $f\in C_c^0(X)$ and $g\in C_c^1(X_\theta)$ we have:
	\begin{align*}
	\la d_\theta f ,g \ra_1=&\frac12\sum_{\e\in A(X)}m(\e)\theta(\e)f(t(\e))\overline{g(\e)}-\frac12\sum_{\e\in A(X)}m(\e)f(o(\e))\overline{g(\e)}\\
	=&-\frac12\sum_{\e\in A(X)}m(\overline{\e})\overline{\theta(\overline{\e})}f(o(\overline{\e}))\theta(\overline{\e})\overline{g(\overline{\e})}-\frac12\sum_{\e\in A(X)}m(\e)f(o(\e))\overline{g(\e)}\\
	=&-\!\sum_{\e\in A(X)}m(\e)f(o(\e))\overline{g(\e)}=\!\sum_{x\in V(X)}m(x)f(x)\Bigl(\overline{-\sum_{\e\in A_x}\frac{m(\e)}{m(x)}g(\e)}\Bigr)=:\la  f,d_\theta^{*}g \ra_0\text{ .}
	\end{align*}
	Simple compositions verify the rest of the lemma.
\end{proof}
 
Note that the magnetic boundary operator is defined by the same formula that the non-magnetic one, but it acts in a different space. As in the previous section, we can extend $d_\theta$ and $d_\theta^*$ by zero and define the \emph{magnetic Gauss--Bonnet operator} $D(X_\theta,m)$ on $l^2(X_\theta,m)$ by:
\begin{equation*}
D(X_\theta,m):=d_\theta+ d_\theta^{*}\text{.}
\end{equation*}
It obviously satisfies
\begin{equation*}
D(X_\theta,m)^2=-\Delta_{0}(X_\theta,m)-\Delta_{1}(X_\theta,m)\text{ .}
\end{equation*}
\subsection{Decomposition into magnetic operators}
As pointed out before the decomposition into a direct integral for the operator $\Delta_0(X,\mg )$ of a weighted topological crystal $(X,\XX,\omega,\Gamma,\mg )$ is well known. It can be implemented by $\U:l^2_0(X,\mg )\to L^2(\hat{\Gamma};l^2_0(\XX,\mg ))$ defined by:
\begin{equation*}
(\U f)(\xi,\xx)=\sum_{\mu\in\Gamma}\overline{\xi(\mu)}f(\mu\hat{\xx})\ ,
\end{equation*}
where $\hat{\Gamma}$ denotes the dual group of $\Gamma$ and the measure $\mg $ is defined unambiguously on $\XX$ by $\mg (\xx)=\mg(\hat{\xx})$ and $\mg (\ee)=\mg(\hat{\ee})$.

Our aim now is to show that $\U$ can be extended to $l^2(X,\mg )$. For every $\xi\in\hat{\Gamma}$ we define
\begin{equation*}
\theta_\xi: A(\XX) \to \T, \quad \theta_\xi(\ee):= \xi\big(\eta(\ee)\big)\ .
\end{equation*}
Recall that $\eta$ was defined in \cref{etadef}. Then one observes that
\begin{equation*}
\theta_\xi(\overline{\ee})=\xi\big(\eta(\overline{\ee})\big)=\xi\big(\eta(\ee)^{-1}\big) = \overline{\xi\big(\eta(\ee)\big)} = \overline{\theta_\xi(\ee)}.
\end{equation*}
It follows that the Hilbert space $l^2(\XX_{\theta_\xi},\mg )$ is well defined for every $\xi\in\hat{\Gamma}$. Since $\Gamma$ is discrete, it is endowed naturally with the counting measure. Hence $\hat{\Gamma}$ has a normalized dual measure denoted by $\d\xi$. We can then define the fibered Hilbert space 
\begin{equation*}
\HH:=\int_{\hat{\Gamma}}^{\oplus}\d\xi\, l^2(\XX_{\theta_\xi},\mg )\ .
\end{equation*}
Obviously $L^2(\hat{\Gamma};l^2_0(\XX,\mg ))\subset \HH$ since the $l^2_0(\XX,\mg )$ is not changed by $\theta_\xi$. We will denote by $\HH_0$ the constant fiber part given by $L^2(\hat{\Gamma};l^2_0(\XX,\mg ))$. We also set for convenience $\HH_1:=\HH\ominus\HH_0$. For an element $\varphi\in \HH$ we will mostly use the notations $\varphi(\xi,\xx):=\[\varphi(\xi)\](\xx)$ and $\varphi(\xi,\ee):=\[\varphi(\xi)\](\ee)$. The unitary equivalence between $l^2(X,\mg )$ and $\HH$ is stated in the next lemma.

\begin{Lemma}\label{Uunitary}
Let $\U:C_c(X)\to \HH$ be defined for all $\xi\in\hat{\Gamma}$, $\xx\in V(\XX)$ and $\ee\in A(\XX)$ by:
\begin{equation*}
(\U f)(\xi,\xx)=\sum_{\mu\in\Gamma}\overline{\xi(\mu)}f(\mu\hat{\xx})\quad ; \quad(\U f)(\xi,\ee)=\sum_{\mu\in\Gamma}\overline{\xi(\mu)}f(\mu\hat{\ee})\text{ .}
\end{equation*}
Then $\U$ extends to a unitary operator, still denoted by $\U$, from $l^2(X,\mg )$ to $\HH$.
\end{Lemma}
\begin{proof}
First we check that $\U$ is well defined. This is, that $\U f(\xi)\in l^2(\XX_{\theta_\xi},\mg )$ for every $\xi\in\hat{\Gamma}$. Indeed one has,
\begin{align*}
\[\U f(\xi)\](\overline{\ee})=&\sum_{\mu\in\Gamma}\overline{\xi(\mu)}f(\mu\hat{\overline{\ee}})\\
=&\sum_{\mu\in\Gamma}\overline{\xi(\mu)}f(\mu\eta(\ee)^{-1}\overline{\hat{\ee}})\\
=&\sum_{\mu\in\Gamma}\overline{\xi(\mu)}\overline{\xi(\eta(\ee))}f(\mu\overline{\hat{\ee}})\\
=&-\overline{\theta_\xi(\ee)}\sum_{\mu\in\Gamma}\overline{\xi(\mu)}f(\mu\hat{\ee})=-\overline{\theta_\xi(\ee)}\[\U f(\xi)\](\ee)\ ,
\end{align*}
where we used the fact $\mu\hat{\overline{\ee}}=\mu\eta(\ee)^{-1}\overline{\hat{\ee}}$.

We can notice that $\U (C_c(X))$ is dense in $\HH$. Indeed, it coincides with elements $\varphi$ such that there exist a family $\{\varphi_\mu\}_{\mu\in\Z^d}\subset l^2(\XX,\mg)$ with only a finite number non identically zero and satisfying for $\xx\in V(\XX)$ and $\xi\in\hat{\Gamma}$:
	\begin{equation*}
\[\varphi(\xi)\](\xx)=\sum_{\mu\in\Z^d}\xi(\mu)\varphi_\mu(\xx)\text{ , }
\end{equation*}
and for $\ee\in A^+(\XX)$
\begin{equation*}
\[\varphi(\xi)\](\ee)=\sum_{\mu\in\Z^d}\xi(\mu)\varphi_\mu(\ee)\text{ and }\[\varphi(\xi)\](\overline{\ee})=\sum_{\mu\in\Z^d}\xi(\mu)\overline{\xi(\eta(\ee))}\varphi_\mu(\overline{\ee})
\ .
	\end{equation*} 
This dense subspace will be denote by $\Trigpol(\HH)=\Trigpol(\HH_0)\oplus\Trigpol(\HH_1)$ by analogy to $L^2(\T^d;\C^{n+l})$. We will now compute $\U^{*}$. Only the calculations corresponding to $l^2_1(X,\mg )$ will be made explicitly. Let $f\in C^1_c(X)$ and $\varphi\in\Trigpol(\HH_1)$. Then
\begin{align*}
\la \U f, \varphi \ra_{\HH_1}=&\int_{\hat{\Gamma}}\d\xi\,\frac12 \sum_{\ee\in A(\XX)}\!\sum_{\mu\in\Gamma}\mg (\ee)\overline{\xi(\mu)}f(\mu\hat{\ee})\overline{\varphi(\xi,\ee)}\\
=&\int_{\hat{\Gamma}}\d\xi\,\frac12\sum_{\e\in A(X)}\mg(\e)\overline{\xi([\e])}f(\e)\overline{\varphi(\xi,\check{\e})}\\
=&\frac12 \sum_{\e\in A(X)} \mg (\e)f(\e)\overline{\int_{\hat{\Gamma}}\d\xi\,\xi([\e])\varphi(\xi,\check{\e})}=:\la f , \U^{*}\varphi\ra_1
\end{align*}
To see that $\U$ is indeed unitary we take $f\in C_c(X)$ and compute
\begin{align*}
\[\U^*\U f\](\e)&=\int_{\hat{\Gamma}}\d\xi\,\xi(\[\e\])\sum_{\mu\in\Gamma}\overline{\xi(\mu)}f(\mu\widehat{\check{\e}})\\
&=\sum_{\mu\in\Gamma}f(\mu\widehat{\check{\e}})\int_{\hat{\Gamma}}\d\xi\,\xi(\[\e\])\overline{\xi(\mu)}=f(\[\e\]\widehat{\check{\e}})=f(\e)\ .
\end{align*}
For $\varphi\in\Trigpol(\HH_1)$
\begin{align*}
\[\U\U^* \varphi\](\xi,\ee)=&\sum_{\mu\in\Gamma}\overline{\xi(\mu)}\int_{\hat{\Gamma}}\!\d\xi'\,\xi'(\[\mu\hat{\ee}\]) \varphi(\xi',\omega(\mu\hat{\ee}))\\
=&\sum_{\mu\in\Gamma}\int_{\hat{\Gamma}}\!\d\xi'\,\overline{\xi(\mu)}\xi'(\mu)\varphi(\xi',\ee)= \varphi(\xi,\ee)\,\text{.}\qedhere
\end{align*}
\end{proof}
For convenience we sum up the formulas for $\U^*$. For $\varphi\in\Trigpol(\HH)$ we have:
\begin{equation*}
(\U^* \varphi)(x)=\int_{\hat{\Gamma}}\d\xi\,\xi([x])\varphi(\xi,\check{x})\text{ ; }(\U^* \varphi)(\e)=\int_{\hat{\Gamma}}\d\xi\,\xi([\e])\varphi(\xi,\check{\e})\text{ .}
\end{equation*}
We now show that by conjugation by $\U$, $H_0$ is unitarily equivalent to the direct integral of magnetic operators. Note that the periodicity of $R_\Gamma$ allows to see it as a multiplication operator on $l^2(\XX_{\theta_\xi},\mg )$ for every $\xi\in\hat{\Gamma}$. 

\begin{Lemma}\label{lemmadecom}
	Let $(X,\XX,\omega,\Gamma,\mg)$ be a weighted topological crystal. Let $H_0$ be defined by \cref{h0def}. Then
	\begin{equation*}
	\U H_0\U^{*}=\int_{\hat{\Gamma}}^{\oplus}\!\!\d\xi\, (D_{\theta_\xi}(\XX,\mg) +R_\Gamma)\text{ .}
	\end{equation*}
\end{Lemma}
\begin{proof}
 For $\varphi\in\Trigpol(\HH)$, see the proof of \cref{Uunitary}, we can compute
	\begin{align*}
	&(\U D(X,\mg)\U^* \varphi)(\xi,\ee)\\
	&=\sum_{\mu\in\Gamma}\overline{\xi(\mu)}\!\!\int_{\hat{\Gamma}}\!\!\d\chi\,\Bigl(\chi([t(\mu\hat{\ee})]) \varphi(\chi,\omega(t(\mu\hat{\ee})))\!-\chi([o(\mu\hat{\ee})]) \varphi(\chi,\omega(o(\mu\hat{\ee})))\Bigr)\\
	&=\sum_{\mu\in\Gamma}\overline{\xi(\mu)}\int_{\hat{\Gamma}}\!\d\chi\,\chi(\mu)\Bigl(\chi(\eta(\hat{\ee})) \varphi(\chi,t(\ee))-\varphi(\chi,o(\ee))\Bigr)\\
	&=\int_{\hat{\Gamma}}\!\d\chi\sum_{\mu\in\Gamma}\overline{\xi(\mu)}\chi(\mu)\Bigl(\chi(\eta(\hat{\ee})) \varphi(\chi,t(\ee))- \varphi(\chi,o(\ee))\Bigr)\\
	&=\xi(\eta(\hat{\ee}))\varphi(\xi,t(\ee))-u(\xi,o(\ee))=(d_{\theta_\xi} \varphi)(\xi,\ee)\text{.}
	\end{align*}
	We used the fact that $\varphi \in\Trigpol(\HH)$ in order to change the order of integration, and also the identity $[ t(\mu\hat{\ee})]=\mu\eta(\hat{\ee})$ and $[o(\mu\hat{\ee})]=\mu$. Similarly we have
	\begin{align*}
	(\U D(X,\mg)\U^* \varphi)(\xi,\xx)=&-\sum_{\mu\in\Gamma}\overline{\xi(\mu)}\sum_{\e\in A_{\mu\hat{\xx}}}\frac{\mg (\e)}{\mg (\mu\hat{\xx})}\int_{\hat{\Gamma}}\!\d\chi\,\chi([\e])\varphi(\chi,\check{\e})\\
	=&-\sum_{\ee\in A_{\xx}}\frac{\mg (\ee)}{\mg (\xx)}\int_{\hat{\Gamma}}\!\d\chi\Bigl(\sum_{\mu\in\Gamma}\overline{\xi(\mu)}\chi(\mu)\Bigr)\varphi(\chi,\ee)\\
	=&-\sum_{\ee\in A_{\xx}}\frac{\mg (\ee)}{\mg (\xx)}\varphi(\xi,\ee)=d^*_{\theta_\xi} \varphi(\xi,\xx)\text{ .}
	\end{align*}
Since the periodicity of $R_\Gamma$ implies that $\U R_\Gamma\U^*=\int_{\hat{\Gamma}}^{\oplus}\!\d\xi R_\Gamma$ we get the desired result.
\end{proof}
\begin{Remark}
It is interesting to notice that the assumption that $\XX$ is finite does not play a role in this decomposition. This means that if one can give conditions that ensure that each operator of the family $D_\theta$ or of the family $\Delta_\theta$ has compact resolvent, one could also study periodic graphs over infinite graphs.
\end{Remark}
\subsection{Integral decomposition with a constant fiber}
In this subsection we show that this decomposition can be made into a constant fiber decomposition. In particular we are interested to get the unitary equivalence between $H_0$ and a matrix-valued multiplication operator on $L^2(\T^d;\C^{n+l})$. Here $\T^d$ stands for the $d-$dimensional (flat) torus, \ie $\T^d:=\R^d/\Z^d$. For this some identifications are necessary. Since $\Gamma$ is isomorphic to $\Z^d$, as stated in \cref{gammaz} of \cref{topocrystal}, we know that $\hat{\Gamma}$ is isomorphic to $\T^d$. In fact, we will consider that a basis of $\Gamma$ is chosen
 and then identify $\Gamma$ with $\Z^{d}$, and accordingly $\hat{\Gamma}$ with $\T^d$. As a consequence of these identifications we shall write $\xi(\mu)=e^{2\pi i\,\xi\cdot\mu}$, where $\xi\cdot\mu = \sum_{j=1}^d\xi_j\mu_j$. Note finally that for elements in $\Z^d$ we use the additive notation but we maintain the use of the juxtaposition for the action of $\Z^d$ on $X$ defined via the identification with $\Gamma$.
 
 The second necessary identification is between $l^2(\XX_{\theta_\xi},\mg )$ and $\C^{n+l}$ for every $\xi\in\hat{\Gamma}$. Indeed, since $V(\XX)= \{\xx_1,\dots,\xx_n\}$, and $A(\XX)=\{\ee_1,\overline{\ee}_1,\dots,\ee_l,\overline{\ee}_l\}$ the vector space $l^2(\XX_{\theta_\xi})$ is of dimension $n+l$ for every $\xi\in\hat{\Gamma}$. One sets
 $\I:\HH\to L^2(\T^d;\C^{n+l})$ by
 \begin{equation*}
 \[\I \varphi\](\xi)_j =\begin{cases}
\mg (\xx_j)^{\frac12}\varphi(\xi,\xx_j)&\text{ for }1\leq j \leq n\\
\mg (\ee_{j-n})^\frac12 \varphi(\xi,\ee_{j-n})&\text{ for }n+1\leq j \leq n+l\ ,
 \end{cases}
 \end{equation*}
 for every $\varphi\in\Trigpol(\HH)$ and every $\xi\in\T^{d}$. We denote by $\Trigpol(\T^d;\C^{n+l})$ the subspace of $L^2(\T^d;\C^{n+l})$ composed by functions $\varphi$ that admits
 \begin{equation*}
 \varphi(\xi)=\sum_{\mu\in\Z^d}e^{2\pi i\xi\cdot\mu}\varphi_\mu\text{ ; with }0\ne\varphi_\mu\in\C^{n+l} \text{ for only finitely many }\mu\ .
 \end{equation*} 
 It is a conveniently dense space in $L^2(\T^d;\C^{n+l})$ and coincides with $\I\U(C_c(X))$. The adjoint of $\I$ is given for $u\in\Trigpol(\T^d;\C^{n+l})$, $\xx\in V(\XX)$ and $\ee\in A^+(\XX)$ by 
 \begin{equation*}
\[\I^* u\](\xi,\xx)=\mg (\xx)^{-\frac12}u(\xi)_{\imath(\xx)}\text{ and }\[\I^* u\](\xi,\ee)=\mg (\ee)^{-\frac12}u(\xi)_{\imath(\ee)+n}\ .
 \end{equation*}
To have $\[\I^* u\](\xi,\overline{\ee})=-\overline{\theta_\xi(\ee)}\[\I^* u\](\xi,\ee)$ we need to set for $\ee\in A^-(\XX)$
\begin{equation*}
\[\I^* u\](\xi,\ee)=-\theta_\xi(\ee)\mg(\ee)^{-\frac{1}2} u(\xi)_{\imath(\ee)+n}\ .
\end{equation*}
Clearly $\I$ is unitary, the only non trivial fact being for $\ee\in A^-(\XX)$ that
\begin{equation*}
\[\I^*\I\varphi\](\xi,\ee)=-\theta_\xi(\ee)\mg(\ee)^{-\frac{1}{2}}(\I\varphi)(\xi)_{\imath(\ee)+n}=-\theta_\xi(\ee)\varphi(\xi,\overline{\ee})=\varphi(\xi,\ee)\ .
\end{equation*}

We can now state the main result of this section. We use the notation $\delta_{j\ell}\equiv\delta(j,\ell)$ for the Kronecker delta.
\begin{Proposition}\label{h0}
	Let  $(X,\XX,\omega,\Gamma)$ be a topological crystal and let $\mg $ be a $\Gamma$-periodic measure on $X$.
	Let $R_\Gamma$ be a real $\Gamma$-periodic potential.
	Then $H_0:=D(X,\mg )+R_\Gamma$
	is unitarily equivalent to a matrix-valued multiplication operator in $L^2(\T^d;\C^{n+l})$
	defined by a real analytic function $h_0:\T^d\to M_{n+l}(\C)$. We can describe $h_0$ by blocks in the following way:
	
	For $1\leq j,\ell\leq n$ we have 
	\begin{equation*}
	h_0(\xi)_{j\ell}=R_\Gamma(\xx_j)\delta_{j\ell}
	\end{equation*}
	and for $n+1\leq j,\ell\leq n+l$ we have
	\begin{equation*}
	h_0(\xi)_{j\ell}=R_\Gamma(\ee_{j-n})\delta_{j\ell}\ .
	\end{equation*}
	Furthermore, for $1\leq j \leq n$ and  $n+1\leq\ell\leq n+l$  we have:
	\begin{equation}\label{dstarevaluado}
	h_0(\xi)_{j\ell}=\frac{\mg (\ee_{\ell-n})^{\frac12}}{\mg (\xx_j)^{\frac12}}
	\begin{cases}
	0&\text{ if } o(\ee_{\ell-n})\neq \xx_j\neq t(\ee_{\ell-n})\\
	-1&\text{ if } o(\ee_{\ell-n}) =\xx_j\neq t(\ee_{\ell-n})\\
	e^{-2\pi i\xi\cdot\eta(\ee_{\ell-n})}&\text{ if } o(\ee_{\ell-n})\neq \xx_j= t(\ee_{\ell-n})\\
	-1+e^{-2\pi i\xi\cdot\eta(\ee_{\ell-n})}&\text{ if } o(\ee_{\ell-n})= \xx_j= t(\ee_{\ell-n})\ ,
	\end{cases}
	\end{equation}
	and
	\begin{equation}\label{devaluado}
	h_0(\xi)_{\ell j}=\frac{\mg (\ee_{\ell-n})^{\frac12}}{\mg (\xx_j)^{\frac12}}
	\begin{cases}
	0&\text{ if } o(\ee_{\ell-n})\neq \xx_j\neq t(\ee_{\ell-n})\\
	-1&\text{ if } o(\ee_{\ell-n}) =\xx_j\neq t(\ee_{\ell-n})\\
	e^{2\pi i\xi\cdot\eta(\ee_{\ell-n})}&\text{ if } o(\ee_{\ell-n})\neq \xx_j= t(\ee_{\ell-n})\\
	-1+e^{2\pi i\xi\cdot\eta(\ee_{\ell-n})}&\text{ if } o(\ee_{\ell-n})= \xx_j= t(\ee_{\ell-n})\ .
	\end{cases}
	\end{equation}
\end{Proposition}
\begin{proof}
Simple computations give for $u\in \Trigpol(\T^d;\C^{n+l})$, $\xi\in\T^d$ and $1\leq j \leq n$
\begin{equation}\label{dstarxi}
(\I D(\XX_{\theta_\xi},\mg )\I^*u)(\xi)_j=\!\!\!\sum_{\ee\in A^-_{\xx_j}}\!\frac{\mg (\ee)^\frac12}{\mg (\xx_j)^\frac12}e^{2\pi i\xi\cdot\eta(\ee)}u(\xi)_{\imath(\ee)+n}-\!\!\!\sum_{\ee\in A^+_{\xx_j}}\!\frac{\mg (\ee)^\frac12}{\mg (\xx_j)^\frac12}u(\xi)_{\imath(\ee)+n}
\end{equation}
while for $n+1\leq \ell \leq n+l$ we have
\begin{multline}\label{dxi}
(\I D(\XX_{\theta_\xi},\mg )\I^*u)(\xi)_\ell\\
=\frac{\mg (\ee_{\ell-n})^\frac12}{\mg (t(\ee_{\ell-n}))^\frac12}e^{2\pi i\xi\cdot\eta(\ee_{\ell-n})}u(\xi)_{\imath(t(\ee_{\ell-n}))}-\frac{\mg (\ee_{\ell-n})^\frac12}{\mg (o(\ee_{\ell-n}))^\frac12}u(\xi)_{\imath(o(\ee_{\ell-n}))}\ .
\end{multline}
Let $\delta_\ell$ denote an element of the canonical basis of $\C^{n+l}$. If we replace $u$ by the constant function $\delta_\ell$ in \cref{dstarxi} we get that it is non zero in two cases: if either $\ee_{\ell-n}$ or $\overline{\ee}_{\ell-n}$ are in $A_{\xx_j}$. In fact, $\ee_{\ell-n}\in A_{\xx_j}$ corresponds to the second line in \cref{dstarevaluado}, $\overline{\ee}_{\ell-n}\in A_{\xx_j}$ to the third line, where the fact that $\overline{\ee}_{\ell-n}\in A^-(\XX)$ explain the negative factor in the exponential. The fourth line corresponds to the case where $\ee_{\ell-n}$ is a loop over $\xx_j$. The same reasoning permit to derive \cref{devaluado} from \cref{dxi}.

It is easy to see that
\begin{equation*}
(\I R_\Gamma \I^* u)(\xi)_{j}=R_\Gamma(\xx_j)u(\xi)_j\text{ for } 1\leq j \leq n
\end{equation*}
and 
\begin{equation*}
(\I R_\Gamma \I^* u)(\xi)_{j}=R_\Gamma(\ee_{j-n})u(\xi)_j \text{ for } n+1\leq j \leq n+l\ .
\end{equation*}
Taking \cref{lemmadecom} into account one gets that $\I\U H_0\U^*\I^*=h_0$. 
\end{proof}
 
\section{Proof of the main theorem}\label{sec:proof}
In this section we provide the proof of \cref{principal}. We summarize first the abstract results obtained in \cite{PR16x}. For a suitable symbol $a:\T^d\times\Z^d\to M_k(\C)$, the corresponding \emph{toroidal pseudo-differential operator} $\Op(a)$ acting $u\in C^\infty(\T^d;\C^k)$ is given by
\begin{equation*}
\[\Op(a)u\](\xi):=\sum_{\mu\in\Z^d}e^{-2\pi i\xi\cdot\mu}a(\xi,\mu)\check{u}(\mu),
\end{equation*}
where $\check{u}$ stands for the inverse Fourier transform defined by 
\begin{equation*}
\[\F^*u\](\mu)\equiv\check{u}(\mu)=\int_{\T^d}\d\xi e^{2\pi i\xi\cdot\mu}u(\xi)\ .
\end{equation*}
In fact, we will mostly be interested in a very simple class of symbols. Let $\nu\in\Z^d$ and $b:\Z^d\to M_k(\C)$ going to zero at infinity be fixed. We set the symbol $b_\nu(\xi,\mu)=e^{2\pi i \xi\cdot\nu}b(\mu)$ and define $b^\dagger_\nu(\xi,\mu)=e^{-2\pi i\xi\cdot\nu}b(\mu+\nu)^*$. Then $\Op(b_\nu)\in\B(L^2(\T^d;\C^k))$ and they satisfy $\Op(b_\nu)^*=\Op(b_\nu^\dagger)$.

Let $\delta_j$ denotes an element of the canonical basis of $\Z^d$. For a function $c:\Z^d\to\R$ we define $\Delta_jf(\mu):=f(\mu+\delta_j)-f(\mu)$. Combining \cite[Theo. 5.7, Lemma 6.2, Lemma 6.3 \& Lemma 6.5]{PR16x} with the classical pertubative Mourre theory, see in particular \cite[Theo. 7.5.6.]{ABG96}, one gets:
\begin{Theorem}\label{abstracto}
Let $h_0:\T^d\to M_k(\C)$ be hermitian-valued and real analytic. Let $h$ be a self-adjoint operator on $L^2(\T^d;\C^k)$ such that the difference $h-h_0$ is a toroidal pseudo-differential operator with symbol $b$. Furthermore, assume that there exists a finite set $\FF$ indexing a family of functions $b(\mathfrak{f}):\T^d\times\Z^d\to M_k(\C)$, a set $\{\nu_{\mathfrak{f}}\}\subset\Z^d$ and $c:\Z^d\to \R$ such that $b$ can be written as
\begin{equation}\label{absdecompuesto}
b=\(\sum_{\mathfrak{f}\in\FF} b(\mathfrak{f})_{\nu_{\mathfrak{f}}}+b(\mathfrak{f})^\dagger_{\nu_{\mathfrak{f}}}\)+c\Id_k\ .
\end{equation}
Assume that each $b(\mathfrak{f})$ satisfies
\begin{equation}\label{condicionshort}
\int_{1}^{\infty}\d\lambda \sup_{\lambda<|\mu|<2\lambda}\lp \[b(\mathfrak{f})\](\mu)\rp<\infty\ ,
\end{equation}
and that $c$ satisfies $\lim_{|\mu|\to \infty}c(\mu)=0$, and for every $j\in\{1,\dots,d\}$
\begin{equation}\label{condicionlong}
\int_{1}^{\infty}\d{\lambda}\sup_{\lambda<|\mu|<2\lambda}|(\triangle_j c)(\mu)|<\infty\ .
\end{equation}
Then there exists a discrete set $\tau\subset\R$ such that for every closed interval $I\subset\R\backslash\tau$ the following assertions hold in I:
\begin{enumerate}
	\item $h_0$ has purely absolutely continuous spectrum,
	\item $h$ has not singular continuous spectrum and has at most a finite number of eigenvalues, each of finite multiplicity,
	\item if $c\equiv 0$ the local wave operators $W_{\pm}(h,h_0;I)$ exist and are asymptotically complete.
\end{enumerate}
\end{Theorem}
\begin{proof}[Proof of \cref{principal}]
By \cref{h0} we know that $\I\U H_0 \U^*\I^*=:h_0$ is a multiplication operator by a real analytic matrix valued function. We ought to show that the difference $\I\U(H- H_0) \U^*\I^*$ satisfies the hypothesis of \cref{abstracto}. For this, we set $\FF=A^+(\XX)\cup\{\ff_0,\ff_s\}$. In this proof $j$ will denote an entire number in $\[1,n\]$ while $\ell$ will be an entire number in $\[n+1,n+l\]$. We start by defining
\begin{equation*}
b(\ff_0)(\mu)_{j\ell}=\begin{cases}
\frac{\mg(\ee_{\ell-n})^\frac12}{\mg(\xx_j)^\frac12}-\frac{m(\mu\hat{\ee}_{\ell-n})^\frac12}{m(\mu\hat{\xx}_j)^\frac12}&\text{ if }o(\ee_{\ell-n})=\xx_j\\
0&\text{ if not,}
\end{cases}
\end{equation*}
and $\nu_{\ff_0}=\mathbf{0}$. For $\ee\in A^+(\XX)$ we define
\begin{equation*}
b(\ee)(\mu)_{j\ell}=\begin{cases}
\frac{m((\mu+\eta(\ee))\hat{\overline{\ee}})^\frac12}{m((\mu+\eta(\ee))\widehat{t(\ee)})^\frac12}-\frac{\mg(\ee)^\frac12}{\mg(t(\ee))^\frac12}&\text{ if }t(\ee)=\xx_j\ , \imath(\ee)=\ell-n\\
0&\text{ if not,}\end{cases}
\end{equation*}
 and $\nu_\ee=-\eta(\ee)$. We also set the diagonal valued symbol $b(\ff_s)$ by
\begin{align*}
b(\ff_s)(\mu)_{jj}=&{}\big(R_S(\mu\hat{\xx}_j)+R_L(\mu\hat{\xx}_j)-R_L(\mu\hat{\xx}_1)\big)/2\\
b(\ff_s)(\mu)_{\ell\ell}=&{}\big(R_S(\mu\hat{\ee}_{\ell-n})+R_L(\mu\hat{\ee}_{\ell-n})-R_L(\mu\hat{\xx}_1)\big)/2
\end{align*}
 with $\nu_{\ff_s}=\mathbf{0}$. Finally we set
 \begin{equation*}
c(\mu)=R_L(\mu\hat{\xx_1})\ .
 \end{equation*} 
 We claim that
 \begin{equation}\label{claim}
\I\U(H- H_0) \U^*\I^*=\sum_{\ff\in\FF}(\Op(b(\ff)_{\nu_\ff})+\Op(b(\ff)^\dagger_{\nu_\ff}))+\Op(c\Id_{n+l})\ .
 \end{equation}
 To show this, we first study $T_1=\I\U(\J D(X,m)\J^*-D(X,\mg)) \U^*\I^*$. We start by noticing that
\begin{align*}
\[(\J D(X,m)\J^*-D(X,\mg))f\](x)&=\sum_{\e\in A_x}\(\frac{\mg(\e)}{\mg(x)}-\frac{\mg(\e)^\frac12 m(\e)^\frac12}{\mg(x)^\frac12 m(x)^\frac12}\)f(\e)\ ,\\
\[(\J D(X,m)\J^*-D(X,\mg))f\](\e)&=\(\frac{\mg(t(\e))^\frac12 m(\e)^\frac12}{\mg(\e)^\frac12 m(t(\e))^\frac12}-1\)f(t(\e))\\
&\quad+\(1-\frac{\mg(o(\e))^\frac12 m(\e)^\frac12}{\mg(\e)^\frac12 m(o(\e))^\frac12}\)f(o(\e))\ .
\end{align*}
Then for $1\leq j \leq n$,
\begin{align}
\[T_1u\](\xi)_j\notag&={}\mg(\xx_j)^\frac12\sum_{\mu\in\Z^d}e^{-2\pi i\xi\cdot\mu}((\J D(X,m)\J^*-D(X,\mg))\U^*\J^*u)(\mu\hat{\xx}_j)\notag\\
&={}\mg(\xx_j)^\frac12\sum_{\mu\in\Z^d}e^{-2\pi i\xi\cdot\mu}\sum_{\e\in A_{\mu\hat{\xx}_j}}\!\!\(\frac{\mg(\e)}{\mg(\mu\hat{\xx}_j)}-\frac{\mg(\e)^\frac12 m(\e)^\frac12}{\mg(\mu\hat{\xx}_j)^\frac12 m(\mu\hat{\xx}_j)^\frac12}\)(\U^*\I^*u)(\e)\notag\\
&={}\sum_{\mu\in\Z^d}e^{-2\pi i\xi\cdot\mu}\[\sum_{\ee\in A^+_{\xx_j}}\(\frac{\mg(\ee)^\frac12}{\mg(\xx_j)^\frac12}-\frac{m(\mu\hat{\ee})^\frac12}{m(\mu\hat{\xx}_j)^\frac12}\)\check{u}(\mu)_{\imath(\ee)+n}\right.\notag\\
&\quad\left.-\sum_{\ee\in A^-_{\xx_j}}\(\frac{\mg(\ee)^\frac12}{\mg(\xx_j)^\frac12}-\frac{m(\mu\hat{\ee})^\frac12}{m(\mu\hat{\xx}_j)^\frac12}\)\int_{\T^d}\d\chi e^{2\pi i\chi\cdot\mu}e^{2\pi i\chi\cdot\eta(\ee)}u(\chi)_{\imath(\ee)+n}\]\notag\\
\begin{split}\label{T1}
&={}\sum_{\mu\in\Z^d}e^{-2\pi i\xi\cdot\mu}\hspace{-5pt}\sum_{\ee\in A^+_{\xx_j}}\hspace{-5pt}\(\frac{\mg(\ee)^\frac12}{\mg(\xx_j)^\frac12}-\frac{m(\mu\hat{\ee})^\frac12}{m(\mu\hat{\xx}_j)^\frac12}\)\check{u}(\mu)_{\imath(\ee)+n}\\
&\quad+\sum_{\mu\in\Z^d}e^{-2\pi i\xi\cdot\mu}\hspace{-5pt}\sum_{\ee\in A^-_{\xx_j}}\hspace{-5pt}\(\frac{m((\mu-\eta(\ee))\hat{\ee})^\frac12}{m((\mu-\eta(\ee))\hat{\xx}_j)^\frac12}-\frac{\mg(\ee)^\frac12}{\mg(\xx_j)^\frac12}\)e^{2\pi i \xi\cdot\eta(\ee)}\check{u}(\mu)_{\imath(\ee)+n}
\end{split}
\end{align}
while for $n<\ell\leq n+l$,
\begin{align}
&\[T_1u\](\xi)_\ell\\
&={}\mg(\ee_{\ell-n})^\frac12\sum_{\mu\in\Z^d}e^{-2\pi i\xi\cdot\mu}((\J D(X,m)\J^*-D(X,\mg))\U^*\J^*u)(\mu\hat{\ee}_{\ell-n})\notag\\
&={}\mg(\ee_{\ell-n})^\frac12\sum_{\mu\in\Z^d}e^{-2\pi i\xi\cdot\mu}\[\(\frac{\mg(t(\ee_{\ell-n}))^\frac12 m(\mu\hat{\ee}_{\ell-n})^\frac12}{\mg(\ee_{\ell-n})^\frac12 m(t(\mu\hat{\ee}_{\ell-n}))^\frac12}-1\)(\U^*\J^*u)(t(\mu\hat{\ee}_{\ell-n}))\right.\notag\\
&\hspace{140pt}\left.+\(1-\frac{\mg(o(\ee_{\ell-n}))^\frac12 m(\mu\hat{\ee}_{\ell-n})^\frac12}{\mg(\ee_{\ell-n})^\frac12 m(o(\mu\hat{\ee}_{\ell-n}))^\frac12}\)(\U^*\J^*u)(o(\mu\hat{\ee}_{\ell-n}))\]\notag\\
&={}\sum_{\mu\in\Z^d}e^{-2\pi i\xi\cdot\mu}\[\(\frac{m(\mu\hat{\ee}_{\ell-n})^\frac12}{m(t(\mu\hat{\ee}_{\ell-n}))^\frac12}-\frac{\mg(\ee_{\ell-n})^\frac12}{\mg(t(\ee_{\ell-n}))^\frac12}\)\hspace{-2pt}\int_{\T^d}\d\chi e^{2\pi i\chi\cdot\mu}e^{2\pi i\chi\cdot\eta(\ee_{\ell-n})}u(\chi)_{\imath(t(\ee_{\ell-n}))}\right.\notag\\
&\hspace{140pt}\left.+\(\frac{\mg(\ee_{\ell-n})^\frac12}{\mg(o(\ee_{\ell-n}))^\frac12}-\frac{m(\mu\hat{\ee}_{\ell-n})^\frac12}{m(o(\mu\hat{\ee}_{\ell-n}))^\frac12}\)\check{u}(\mu)_{\imath(o(\ee_{\ell-n}))}\]\notag\\
\begin{split}\label{T1j}
&=\hspace{-3pt}\sum_{\mu\in\Z^d}\hspace{-3pt}e^{-2\pi i\xi\cdot\mu}\hspace{-3pt}\(\hspace{-3pt}\frac{m((\mu-\eta(\ee_{\ell-n}))\hat{\ee}_{\ell-n})^\frac12}{m((\mu-\eta(\ee_{\ell-n}))t(\hat{\ee}_{\ell-n}))^\frac12}\hspace{-1pt}-\hspace{-1pt}\frac{\mg(\ee_{\ell-n})^\frac12}{\mg(t(\ee_{\ell-n}))^\frac12}\hspace{-3pt}\)\hspace{-3pt}e^{2\pi i\xi\cdot\eta(\ee_{\ell-n})}\check{u}(\mu)_{\imath(t(\ee_{\ell-n}))}\\
&\quad+\hspace{-3pt}\sum_{\mu\in\Z^d}e^{-2\pi i\xi\cdot\mu}\(\frac{\mg(\ee_{\ell-n})^\frac12}{\mg(o(\ee_{\ell-n}))^\frac12}-\frac{m(\mu\hat{\ee}_{\ell-n})^\frac12}{m(o(\mu\hat{\ee}_{\ell-n}))^\frac12}\)\check{u}(\mu)_{\imath(o(\ee_{\ell-n}))}\ .
\end{split}
\end{align}
We can now deduce that $\Op(b(\ff_0)_{\nu_{\ff_0}})$ corresponds to the first term of \cref{T1} while $\Op(b(\ff_0)^\dagger_{\nu_{\ff_0}})$ correspond to the second term of \cref{T1j}. 

For $\Op(b(\ee)_{\nu_{\ee}})$ we compute
\begin{align*}
&\[\Op(b(\ee)_{\nu_\ee})u\](\xi)_j\\
&=\sum_{\mu\in\Z^d}e^{-2\pi i\xi\cdot\mu}\sum_{\ell=1}^{n+l}\[b(\ee)_{\nu_\ee}(\xi,\mu)\]_{j\ell}\check{u}(\mu)_\ell\\
&=\sum_{\mu\in\Z^d}e^{-2\pi i\xi\cdot\mu}\(\frac{m((\mu+\eta(\ee))\hat{\overline{\ee}})^\frac12}{m((\mu+\eta(\ee))\widehat{t(\ee)})^\frac12}-\frac{\mg(\ee)^\frac12}{\mg(t(\ee))^\frac12}\)e^{-2\pi i\xi\cdot\eta(\ee)}\delta(\imath(t(\ee)),j)\check{u}(\mu)_{\imath(\ee)+n}\\
&=\sum_{\mu\in\Z^d}e^{-2\pi i\xi\cdot\mu}\(\frac{m((\mu-\eta(\overline{\ee}))\hat{\overline{\ee}})^\frac12}{m((\mu-\eta(\overline{\ee}))\widehat{o(\overline{\ee})})^\frac12}-\frac{\mg(\overline{\ee})^\frac12}{\mg(o(\overline{\ee}))^\frac12}\)e^{2\pi i\xi\cdot\eta(\overline{\ee})}\delta(\imath(o(\overline{\ee})),j)\check{u}(\mu)_{\imath(\overline{\ee})+n}\ .
\end{align*}
Then we can see each $\Op(b(\ee)_{\nu_{\ee}})$ corresponds to a term of the sum over $A^-(\XX)$ in the second term of \cref{T1}. In order to see that $\Op(b(\ee)^\dagger_{\nu_{\ee}})$
corresponds to the remaining term, we need to notice that $\hat{\overline{\ee}}=\overline{-\eta(\ee)\hat{\ee}}$ and $-\eta(\ee)t(\hat{\ee})=\widehat{t(\ee)}$. Then we can compute
	\begin{align*}
b(\ee)^\dagger_{\nu_{\ee}}(\xi,\mu)_{\ell j}={}&e^{-2\pi i\xi\cdot\nu_\ee}b(\ee)(\mu+\nu_\ee)_{j\ell}\\
={}&e^{2\pi i\xi\cdot\eta(\ee)}\begin{cases}
\frac{m(\mu\hat{\overline{\ee}})^\frac12}{m(\mu\widehat{t(\ee)})^\frac12}-\frac{\mg(\ee)^\frac12}{\mg(t(\ee))^\frac12}&\text{ if }t(\ee)=\xx_j\ ,\imath(\ee)=\ell-n\\
0&\text{ if not}
\end{cases}\\
={}&e^{2\pi i\xi\cdot\eta(\ee)}\begin{cases}
\frac{m((\mu-\eta(\ee))\hat{\ee})^\frac12}{m((\mu-\eta(\ee))t(\hat{\ee}))^\frac12}-\frac{\mg(\ee)^\frac12}{\mg(t(\ee))^\frac12}&\text{ if }t(\ee)=\xx_j\ ,\imath(\ee)=\ell-n\\
0&\text{ if not.}
\end{cases}
	\end{align*}
It follows that
\begin{align*}
&\[\Op(b(\ee)^\dagger_{\nu_\ee})u\](\xi)_\ell\\
&=\sum_{\mu\in\Z^d}e^{-2\pi i\xi\cdot\mu}\sum_{j=1}^{n+l}\[b(\ee)^\dagger_{\nu_\ee}(\xi,\mu)\]_{\ell j}\check{u}(\mu)_j\\
&=\sum_{\mu\in\Z^d}e^{-2\pi i\xi\cdot\mu}\(\frac{m((\mu-\eta(\ee))\hat{\ee})^\frac12}{m((\mu-\eta(\ee))t(\hat{\ee}))^\frac12}-\frac{\mg(\ee)^\frac12}{\mg(t(\ee))^\frac12}\)e^{2\pi i\xi\cdot\eta(\ee)}\delta(\imath(\ee),\ell-n)\check{u}(\mu)_{\imath(t(\ee))}\ .
\end{align*}
We can then conclude that
\begin{equation}\label{Dsimbolo}
T_1=\Op(b(\ff_0)_{\nu_{\ff_0}})+\Op(b(\ff_0)^\dagger_{\nu_{\ff_0}})+\sum_{\ee\in A^+(\XX)}\Op(b(\ee)_{\nu_{\ee}})+\Op(b(\ee)^\dagger_{\nu_{\ee}})\ .
\end{equation}
Noticing that $b(\ff_s)_{\nu_{\ff_s}}=b(\ff_s)_{\nu_{\ff_s}}^\dagger$, we get that
\begin{align*}
\big(b(\ff_s)_{\nu_{\ff_s}}(\xi,\mu)+b(\ff_s)_{\nu_{\ff_s}}^\dagger(\xi,\mu)+c(\mu)\Id_{n+l}\big)_{jj}={}&R_S(\mu\hat{\xx}_j)+R_L(\mu\hat{\xx}_j)\\
\big(b(\ff_s)_{\nu_{\ff_s}}(\xi,\mu)+b(\ff_s)_{\nu_{\ff_s}}^\dagger(\xi,\mu)+c(\mu)\Id_{n+l}\big)_{\ell\ell}={}&R_S(\mu\hat{\ee}_{\ell-n})+R_L(\mu\hat{\ee}_{\ell-n})\,
\end{align*}
and hence
\begin{equation}\label{Rsimbolo}
\I\U(R_S+R_L)\U^*\I^*=\Op(b(\ff_s)_{\nu_{\ff_s}})+\Op(b(\ff_s)^\dagger_{\nu_{\ff_s}})+\Op(c\Id_{n+l})\ .
\end{equation}
Taking into account \cref{Dsimbolo} and \cref{Rsimbolo}, we get \cref{claim}.

To see that each $b(\ee)$, with $\ee\in A^+(\XX)$, satisfies \cref{condicionshort} one needs to take into account \cref{mshortD} and the relation
\begin{equation*}
\lp b(\ee)(\mu)\rp=\left|\frac{m((\mu+\eta(\ee))\hat{\overline{\ee}})}{m((\mu+\eta(\ee))\widehat{t(\ee)})}-\frac{\mg(\ee)}{\mg(t(\ee))}\right|\times\left|\frac{m((\mu+\eta(\ee))\hat{\overline{\ee}})^\frac12}{m((\mu+\eta(\ee))\widehat{t(\ee)})^\frac12}+\frac{\mg(\ee)^\frac12}{\mg(t(\ee))^\frac12}\right|^{-1}\ .
\end{equation*}
Then one only needs to notice that the second term is uniformly bounded as a function of $\mu$. A similar argument applied to each entry of $b(\ff_0)$ shows that it also satisfies \cref{condicionshort}.

In order to treat $b(\ff_s)$ we will need a property of paths. Let $\alpha=\{\e_p\}_{p=1}^N$ be a path between $x$ and $y$ in $X$. This means that $o(\e_1)=x$, $t(\e_p)=o(\e_{p+1})$ and $t(\e_N)=y$. Then, for any potential $R$ on $X$ the following formula holds:
\begin{equation*}
R(y)-R(x)=\sum_{\e\in\alpha}\(R(t(\e))-R(\e)\)+\sum_{\e\in\alpha}\(R(\e)-R(o(\e))\)\ .
\end{equation*}
Let $1< j\leq n$ be fixed. Let us denote by $\alpha_j$ a fixed path between $x_1=\hat{\xx}_1$ and $x_j=\hat{\xx}_j$. By applying $\mu$ to each edge in $\alpha_j$ we get a path between $\mu x_1$ and $\mu x_j$. It follows that:
\begin{align*}
|2b(\ff_s)(\mu)_{jj}|={}&|R_S(\mu x_j)+R_L(\mu x_j)-R_L(\mu x_1)|\\
\leq{}&|R_S(\mu x_j)|+\sum_{\e\in\alpha_j}|R_L(t(\mu\e))-R_L(\mu\e)|+\sum_{\e\in\alpha_j}|R_L(\mu\e)-R_L(o(\mu\e))|\ .
\end{align*}
Then by combining \cref{potshort} and \cref{longGauss} we get that
\begin{equation*}
\int_1^\infty\d\lambda\sup_{\lambda<|\mu|<2\lambda}|b(\ff_s)(\mu)_{jj}|<\infty \text{ for }1\leq j\leq n\ .
\end{equation*}
For $n<\ell\leq n+l$ a similar argument holds. We need to fix a path $\alpha_\ell$ between $x_1$ and $o(\e_{\ell-n})$ and compute
\begin{align*}
|2b(\ff_s)(\mu)_{\ell\ell}|={}&|R_S(\mu \e_{\ell-n})+R_L(\mu \e_{\ell-n})-R_L(\mu x_1)|\\
={}&|R_S(\mu x_\ell)|+|R_L(\mu \e_{\ell-n})-R_L(\mu o(\e_{\ell-n}))|\\&\qquad+\sum_{\e\in\alpha_j}|R(t(\mu\e))-R(\mu\e)|+\sum_{\e\in\alpha_\ell}|R(\mu\e)-R(o(\mu\e))|\ .
\end{align*}
Finally, we only need to check that $c$ fulfills \cref{condicionlong}. For this we fix paths $\beta_j$ between $x_1$ and $\delta_jx_1$ for every $j\in\{1,\dots,d\}$. Then we have
\begin{align*}
|(\Delta_jc)(\mu)|={}&|c(\delta_j+\mu)-c(\mu)|\\
={}&|R_L((\delta_j+\mu)x_1)-R_L(\mu x_1)|\\
\leq{}&\sum_{\e\in\beta_j}|R_L(t(\mu\e))-R_L(\mu\e)|+\sum_{\e\in\beta_j}|R_L(\mu\e)-R_L(o(\mu\e))|\ .
\end{align*}
From this we can see that \cref{longGauss} implies \cref{condicionlong} finishing the proof.
\end{proof}

\section{Schrödinger operators acting on edges}\label{sec:laplacian}
In this section we turn our attention to the operator $\Delta_1(X,\mg)$. It acts on $l^2_1(X,\mg)$ by
\begin{equation*}
\[\Delta_1(X,\mg)f\](\e):=\[-d^*df\](\e)=\sum_{\e'\in A_{t(\e)}}\frac{\mg(\e')}{\mg(t(\e))}f(\e')-\sum_{\e'\in A_{o(\e)}}\frac{\mg(\e')}{\mg(o(\e))}f(\e')\ .
\end{equation*}
Given a periodic potential $R_\Gamma:E(X)\to\R$ we define the periodic Schrödinger operator acting on edges by
\begin{equation}\label{H0edges}
H_0:=-\Delta_1(X,\mg)+R_\Gamma\ .
\end{equation}
The perturbed operator is defined by choosing a non periodic measure $m$ that converges to $\mg$ at infinity. We still denote by $\J$ the restriction of $\J$ to an operator from $l^2_1(X,m)$ to $l^2_1(X,\mg)$. Then $H$ is defined by
\begin{equation}\label{Hedges}
H:=\J(-\Delta_1(X,m))\J^*+R\ ,
\end{equation}
where $R$ is a potential that converges to $R_\Gamma$ at infinity. 
\begin{Theorem}
	Let $X$ be a topological crystal. Let $H_0$ and $H$ be defined by \cref{H0edges} and \cref{Hedges} respectively. Assume that $m$ satisfies
	\begin{equation}\label{measureshort}
	\int_{1}^{\infty}\d\lambda \sup_{\lambda<|\[\e\]|<2\lambda}\left|\frac{m(\e)}{m(o(\e))}-\frac{\mg (\e)}{\mg (o(\e))}\right|<\infty\ .
	\end{equation}
	Assume also that the difference $R-R_\Gamma$ is equal to $R_S+R_L$ which satisfy
	\begin{equation}
	\int_{1}^{\infty}\d\lambda \sup_{\lambda<|\[x\]|<2\lambda}\left|R_S(x)\right|<\infty,
	\end{equation}
	and
	\begin{equation}\label{Rlhipotesis}
	R_L(\e)\xrightarrow{\e\to\infty}0, \quad \hbox{ and } \quad
	\int_{1}^{\infty}\d\lambda \sup_{\lambda<|\[x\]|<2\lambda}\sup_{\e,\e'\in A_x}\big|R_L(\e)-R_L(\e')\big|<\infty\ .
	\end{equation}
	Then, there exists a discrete set $\tau\subset\R$ such that for every closed interval $I\subset\R\backslash\tau$ the following assertions hold in $I$:
	\begin{enumerate}
	\item $H_0$ has purely absolutely continuous spectrum,
	\item $H$ has no singular continuous spectrum and has at most a finite number of eigenvalues, each of finite multiplicity,
	\item if $R_L\equiv0$ the local wave operators $W_{\pm}(H,H_0;I)$ exist and are asymptotically complete.
	\end{enumerate}
\end{Theorem}
\begin{proof}
The proof goes in the lines of the proof of \cref{principal}. We start by noticing that
\begin{align*}
\hspace{-5pt}\[(\Delta_1(X,\mg)\hspace{-1pt}-\hspace{-1pt}\J\Delta_1(X,m)\J^*)f\](\e)&=\hspace{-8pt}\sum_{\e'\in A_{t(\e)}}\hspace{-7pt}\(\frac{\mg(\e')}{\mg(t(\e))}-\frac{\mg(\e')^\frac12m(\e)^\frac12m(\e')^\frac12}{\mg(\e)^\frac12m(t(\e))}\)\!f(\e')\\
&+\hspace{-8pt}\sum_{\e'\in A_{o(\e)}}\hspace{-7pt}\(\frac{\mg(\e')^\frac12m(\e)^\frac12m(\e')^\frac12}{\mg(\e)^\frac12m(o(\e))}-\frac{\mg(\e')}{\mg(o(\e))}\)f(\e')\ .
\end{align*}
We set for convenience $T_2=\I\U(\Delta_1(X,\mg)-\J\Delta_1(X,m)\J^*)\U^*\I^*$, where $\I$ stands for the restriction $\I:\HH_1\to L^2(\T^d;\C^l)$. We can compute for $1\leq \ell\leq l$:
\begin{align}
&\[T_2u\](\xi)_\ell=\mg(\ee_\ell )^\frac12\sum_{\mu\in\Z^d}e^{-2\pi i\xi\cdot\mu}\big((\Delta_1(X,\mg)-\J\Delta_1(X,m)\J^*)\U^*\J^*u\big)(\mu\hat{\ee}_\ell )\nonumber\\
&=\sum_{\mu\in\Z^d}e^{-2\pi i\xi\cdot\mu}\[\sum_{\e\in A_{t(\mu\hat{\ee}_\ell )}}\hspace{-5pt}\(\frac{\mg(\e)\mg(\ee_\ell )^\frac12}{\mg(t(\ee_\ell ))}-\frac{\mg(\e)^\frac12m(\mu\hat{\ee}_\ell )^\frac12m(\e)^\frac12}{m( t(\mu\hat{\ee}_\ell ))}\)(\U^*\J^*u)(\e)\right.\nonumber\\
&\hspace{82pt}\left.+\sum_{\e\in A_{o(\mu\hat{\ee}_\ell )}}\(\frac{\mg(\e)^\frac12m(\mu\hat{\ee}_\ell )^\frac12m(\e)^\frac12}{m(o(\mu\hat{\ee}_\ell ))}-\frac{\mg(\e)^\frac12\mg(\ee_\ell )^\frac12}{\mg(o(\mu\hat{\ee}_\ell ))}\)(\U^*\J^*u)(\e)\]\nonumber\\
&=\sum_{\mu\in\Z^d}e^{-2\pi i\xi\cdot\mu}\[\sum_{\ee\in A_{t(\ee_\ell )}}\(\frac{\mg(\ee)^\frac12\mg(\ee_\ell )^\frac12}{\mg(t(\ee_\ell ))}-\frac{m(\mu\hat{\ee}_\ell )^\frac12m((\mu+\eta(\ee_\ell))\hat{\ee})^\frac12}{m( t(\mu\hat{\ee}_\ell ))}\)\nonumber\right.\\
&\hspace{82pt}\times\int_{\T^d}\d\chi e^{2\pi i\chi\cdot(\mu+\eta(\ee_\ell ))}u(\chi)_{\imath(\ee)}\begin{cases}
1&\text{ if }\ee\in A^+(\XX)\\
-\theta_{\xi}(\ee)&\text{ if }\ee\in A^-(\XX)
\end{cases}\nonumber\\
&\hspace{82pt}+\sum_{\ee\in A_{o(\ee_\ell )}}\(\frac{m(\mu\hat{\ee}_\ell )^\frac12m(\mu\hat{\ee})^\frac12}{m(\mu o(\hat{\ee}_\ell ))}-\frac{\mg(\ee)^\frac12\mg(\ee_\ell )^\frac12}{\mg(o(\ee_\ell ))}\)\nonumber\\
&\hspace{82pt}\times\left.\int_{\T^d}\d\chi e^{2\pi i\chi\cdot\mu}u(\chi)_{\imath(\ee)}\begin{cases}
1&\text{ if }\ee\in A^+(\XX)\\
-\theta_{\xi}(\ee)&\text{ if }\ee\in A^-(\XX)
\end{cases}\vphantom{\sum_{\ee\in A_{t(\ee_\ell )}}\(\frac{\mg(\ee)^\frac12\mg(\ee_\ell )^\frac12}{\mg(t(\ee_\ell ))}-\frac{m(\mu\hat{\ee}_\ell )^\frac12m(\mu\hat{\ee})^\frac12}{m( t(\mu\hat{\ee}_\ell ))}\)}\]\nonumber\\
&=\sum_{\mu\in\Z^d}e^{-2\pi i\xi\cdot\mu}\[\sum_{\ee\in A^-_{t(\ee_\ell )}}\hspace{-7pt}\(\frac{m((\mu-\eta(\ee_\ell )-\eta(\ee))\hat{\ee}_\ell )^\frac12m((\mu-\eta(\ee))\hat{\ee})^\frac12}{m((\mu-\eta(\ee_\ell )-\eta(\ee))t(\hat{\ee}_\ell ))}\right.\right.\nonumber\\
&\left.\hspace{180pt}-\frac{\mg(\ee)^\frac12\mg(\ee_\ell )^\frac12}{\mg(t(\ee_\ell ))}\)e^{2\pi i\xi\cdot(\eta(\ee_\ell )+\eta(\ee))}\check{u}(\mu)_{\imath(\ee)}\label{casoa}\\
&\quad+\hspace{-5pt}\sum_{\ee\in A^+_{t(\ee_\ell )}}\hspace{-5pt}\(\frac{\mg(\ee)^\frac12\mg(\ee_\ell )^\frac12}{\mg(t(\ee_\ell ))}-\frac{m((\mu-\eta(\ee_\ell ))\hat{\ee}_\ell )^\frac12m(\mu\hat{\ee})^\frac12}{m( (\mu-\eta(\ee_\ell ))t(\hat{\ee}_\ell ))}\)e^{2\pi i\xi\cdot\eta(\ee_\ell )}\check{u}(\mu)_{\imath(\ee)}\label{casob}\\
&\quad+\hspace{-5pt}\sum_{\ee\in A^-_{o(\ee_\ell )}}\hspace{-5pt}\(\frac{\mg(\ee)^\frac12\mg(\ee_\ell )^\frac12}{\mg(o(\ee_\ell ))}-\frac{m((\mu-\eta(\ee))\hat{\ee}_\ell )^\frac12m((\mu-\eta(\ee))\hat{\ee})^\frac12}{m((\mu-\eta(\ee)) o(\hat{\ee}_\ell ))}\)e^{2\pi i\xi\cdot\eta(\ee)}\check{u}(\mu)_{\imath(\ee)}\label{casoc}\\
&\quad\left.+\hspace{-5pt}\sum_{\ee\in A^+_{o(\ee_\ell )}}\hspace{-5pt}\(\frac{m(\mu\hat{\ee}_\ell )^\frac12m(\mu\hat{\ee})^\frac12}{m(\mu o(\hat{\ee}_\ell ))}-\frac{\mg(\ee)^\frac12\mg(\ee_\ell )^\frac12}{\mg(o(\ee_\ell ))}\)\check{u}(\mu)_{\imath(\ee)}\]\label{casod}
\end{align}
We need to define a symbol for each one of the lines \cref{casoa} to \cref{casod}. Each one correspond to a different way in which two oriented edges can intersect. Considering the $\delta$ function also defined for vertices by $\delta(\xx,\xx')=\delta(\imath(\xx),\imath(\xx'))$ we can, for every pair $(\ee_j,\ee_\ell)$, define matrices with a single entry in the position $j\ell$ by:
\begin{align*}
\[a(\ee_j,\ee_\ell)\](\mu)_{j\ell}=&\(\frac{m((\mu-\eta(\ee_\ell)+\eta(\ee_j))\hat{\ee}_\ell)^\frac12m((\mu+\eta(\ee_j))\hat{\overline{\ee}}_j)^\frac12}{m((\mu-\eta(\ee_\ell)+\eta(\ee_j))t(\hat{\ee}_\ell))}\right.\nonumber\\
&\hspace{176pt}\left.-\frac{\mg(\ee_j)^\frac12\mg(\ee_\ell)^\frac12}{\mg(t(\ee_\ell))}\)\delta\big(t(\ee_j),t(\ee_\ell)\big)
\\
\[b(\ee_j,\ee_\ell)\](\mu)_{j\ell}=&\(\frac{\mg(\ee_j)^\frac12\mg(\ee_\ell)^\frac12}{\mg(t(\ee_\ell))}-\frac{m((\mu-\eta(\ee_\ell))\hat{\ee}_\ell)^\frac12m(\mu\hat{\ee}_j)^\frac12}{m((\mu-\eta(\ee_\ell))t(\hat{\ee}_\ell))}\)\delta\big(o(\ee_j),t(\ee_\ell)\big)
\\
\[c(\ee_j,\ee_\ell)\](\mu)_{j\ell}=&\(\hspace{-2pt}\frac{\mg(\ee_j)^\frac12\mg(\ee_\ell)^\frac12}{\mg(o(\ee_\ell))}\hspace{-1pt}-\hspace{-1pt}\frac{m((\mu+\eta(\ee_j))\hat{\ee}_\ell)^\frac12m((\mu+\eta(\ee_j))\hat{\overline{\ee}}_j)^\frac12}{m((\mu+\eta(\ee_j))o(\hat{\ee}_\ell))}\hspace{-2pt}\)\hspace{-3pt}\delta\big(t(\ee_j),o(\ee_\ell)\big)
\\
\[d(\ee_j,\ee_\ell)\](\mu)_{j\ell}=&\(\frac{m(\mu\hat{\ee}_j)^\frac12m(\mu\hat{\ee}_\ell)^\frac12}{m(\mu o(\hat{\ee}_\ell))}-\frac{\mg(\ee_j)^\frac12\mg(\ee_\ell)^\frac12}{\mg(o(\ee_\ell))}\)\delta\big(o(\ee_j),o(\ee_\ell)\big)\ .
\end{align*}
Furthermore we set
\begin{align*}
A(\ee_j,\ee_\ell)=&\eta(\ee_\ell)-\eta(\ee_j) &B(\ee_j,\ee_\ell)=&\eta(\ee_\ell)\\
C(\ee_j,\ee_\ell)=&-\eta(\ee_j) &D(\ee_j,\ee_\ell)=&\mathbf{0}\ .
\end{align*}
We can see that
\begin{multline}\label{t2decompuesto}
T_2=\sum_{\ee_j,\ee_\ell\in A^+(\XX)}\[\Op(a(\ee_j,\ee_\ell)_{A(\ee_j,\ee_\ell)})+\Op(b(\ee_j,\ee_\ell)_{B(\ee_j,\ee_\ell)})\right.\\\left.+\Op(c(\ee_j,\ee_\ell)_{C(\ee_j,\ee_\ell)})+\Op(d(\ee_j,\ee_\ell)_{D(\ee_j,\ee_\ell)})\]\ .
\end{multline}
To verify that \cref{t2decompuesto} is of the form \cref{absdecompuesto} we need only to verify
\begin{align}
a(\ee_j,\ee_\ell)_{A(\ee_j,\ee_\ell)}^\dagger=&a(\ee_\ell,\ee_j)_{A(\ee_\ell,\ee_j)}\label{adagger}\\
b(\ee_j,\ee_\ell)_{B(\ee_j,\ee_\ell)}^\dagger=&c(\ee_\ell,\ee_j)_{C(\ee_\ell,\ee_j)}\label{bdagger}\\
d(\ee_j,\ee_\ell)_{D(\ee_j,\ee_\ell)}^\dagger=&d(\ee_\ell,\ee_j)_{D(\ee_\ell,\ee_j)}\label{ddagger}\ .
\end{align}
Indeed, keeping in mind the equality $\hat{\overline{\ee}}=\overline{-\eta(\ee)\hat{\ee}}$, one has
\begin{align*}
&\[a(\ee_j,\ee_\ell)_{A(\ee_j,\ee_\ell)}^\dagger\](\xi,\mu)_{\ell j}=\\
&=e^{-2\pi i\xi\cdot A(\ee_j,\ee_\ell)}a(\ee_j,\ee_\ell)_{j\ell}(\mu+ A(\ee_j,\ee_\ell))\\
&=e^{-2\pi i\xi\cdot (\eta(\ee_\ell)-\eta(\ee_j))}\(\frac{m(\mu\hat{\ee}_\ell)^\frac12m((\mu+\eta(\ee_\ell))\hat{\overline{\ee}}_j)^\frac12}{m(\mu t(\hat{\ee}_\ell))}-\frac{\mg(\ee_\ell)^\frac12\mg(\ee_j)^\frac12}{\mg(t(\ee_\ell))}\)\delta\big(t(\ee_j),t(\ee_\ell)\big)\\
&=e^{2\pi i\xi\cdot (\eta(\ee_j)-\eta(\ee_\ell))}\hspace{-2pt}\(\hspace{-2pt}\frac{m(\mu(\overline{\eta(\ee_\ell)\hat{\overline{\ee}}_\ell}))^\frac12m((\mu+\eta(\ee_\ell))(\overline{-\eta(\ee_j)\hat{\ee}_j}))^\frac12}{m((\mu+\eta(\ee_\ell))\widehat{t(\ee_\ell)})}\right.\\
&\hspace{259pt}\left.-\frac{\mg(\ee_j)^\frac12\mg(\ee_\ell)^\frac12}{\mg(t(\ee_j))}\hspace{-2pt}\)\hspace{-2pt}\delta\big(t(\ee_\ell),t(\ee_j)\big)\\
&=e^{2\pi i\xi\cdot (\eta(\ee_j)-\eta(\ee_\ell))}\(\frac{m((\mu+\eta(\ee_\ell))\hat{\overline{\ee}}_\ell)^\frac12m((\mu-\eta(\ee_j)+\eta(\ee_\ell))\hat{\ee}_j)^\frac12}{m((\mu+\eta(\ee_\ell)-\eta(\ee_j))t(\hat{\ee}_j))}\right.\\
&\hspace{259pt}\left.-\frac{\mg(\ee_j)^\frac12\mg(\ee_\ell)^\frac12}{\mg(t(\ee_j))}\hspace{-2pt}\)\hspace{-2pt}\delta\big(t(\ee_\ell),t(\ee_j)\big)\\
&=e^{2\pi i\xi\cdot A(\ee_\ell,\ee_j)}a(\ee_\ell,\ee_j)_{\ell j}(\mu)=\[a(\ee_\ell,\ee_j)_{A(\ee_\ell,\ee_j)}\](\xi,\mu)_{\ell j}\ ,
\end{align*}
getting \cref{adagger}. Similar computations give \cref{bdagger} while \cref{ddagger} is straightforward. 

It remains to show that each symbol 
fulfill the decay \cref{condicionshort}. We only do it for the symbol $d$. Let us fix $\ee_j$ and $\ee_\ell$ and we set for any $\mu\in\Z^d$: $f(\mu):=\frac{m(\mu\hat{\ee}_j)}{m(\mu o(\hat{\ee}_j)}$, $g(\mu):=\frac{m(\mu\hat{\ee}_\ell)}{m(\mu o(\hat{\ee}_\ell)}$, $f_0=\frac{\mg(\ee_j)}{\mg(o(\ee_j))}$ and $g_0=\frac{\mg(\ee_\ell)}{\mg(o(\ee_\ell))}$. 
Then one deduces that
\begin{align*}
\lp \[d(\ee_j,\ee_\ell)\](\mu)\rp
&=\left|f(\mu)^\frac12 g(\mu)^\frac12 -f_0^\frac12 g_0^\frac12\right|\\
&=\left|\big(f(\mu)-f_0\big)\frac{g(\mu)^\frac12}{f(\mu)^\frac12+f_0^\frac12}
+\big(g(\mu)-g_0\big)\frac{f_0^\frac12}{g(\mu)^\frac12+g_0^\frac12}\right|\ .
\end{align*}
Since the functions
$\frac{g^\frac12}{f^\frac12+f_0^\frac12}$ and $\frac{f_0^\frac12}{g^\frac12+g_0^\frac12}$ are bounded on $\Z^d$  we finally obtain
\begin{align*}
\sup_{\lambda<|\mu|<2\lambda }\lp \[d(\ee_j,\ee_\ell)\](\mu)\rp
& \leq C \Big( \sup_{\lambda<|\mu|<2\lambda }|f(\mu)-f_0|+\sup_{\lambda<|\mu|<2\lambda }|g(\mu)-g_0|\Big) \\
& \leq C \Big( \sup_{\lambda<|\mu|<2\lambda }\Big|\frac{m(\mu\hat{\ee}_j)}{m(\mu o(\hat{\ee}_j)}-\frac{\mg(\ee_j)}{\mg(o(\ee_j))}\Big| \\
& \qquad \qquad + \sup_{\lambda<|\mu|<2\lambda }\Big|\frac{m(\mu\hat{\ee}_\ell)}{m(\mu o(\hat{\ee}_\ell)}
-\frac{\mg(\ee_\ell)}{\mg(o(\ee_\ell))}\Big|\Big).
\end{align*}
Hence \cref{condicionshort} is direct consequence of \cref{measureshort}. Similar computations yield the same property for the symbols $a,b,c$ once one takes into account the invariance of condition \cref{measureshort} under a finite shift.

To finish the proof we need only to treat the multiplicative perturbation $R-R_\Gamma=R_S+R_L$. The proof is the same that for \cref{principal} setting
\begin{equation*}
\[b(\ff_s)\](\mu)_{jj}=R_S(\mu\hat{\ee}_j)+R_L(\mu\hat{\ee}_j)-R_L(\mu\hat{\ee}_1)\ ,
\end{equation*}
and
\begin{equation*}
c(\mu)=R_L(\mu\hat{\ee}_1)\ .
\end{equation*}
Then, to show the short range condition on $R_L(\mu\hat{\ee}_j)-R_L(\mu\hat{\ee}_1)$ for $1\leq j\leq l$ one needs to fix paths $\alpha_j=\{\e_{j,p}\}_{p=1}^N$ between $t(\hat{\ee}_1)$ and $o(\hat{\ee}_j)$ and compute:
\begin{multline*}
|R_L(\mu\hat{\ee}_j)-R_L(\mu\hat{\ee}_1)|=|(R_L(\mu\hat{\ee}_j)-R_L(\mu\e_{j,N}))+\sum_{p=2}^{N}(R_L(\mu\e_{j,p})-R_L(\mu\e_{j,p-1}))\\
\hspace{50pt}+(R_L(\mu\e_{j,1}))-R_L(\mu\hat{\ee}_1)|
\end{multline*}
Each term has the required decay \cref{condicionshort} as a consequence of \cref{Rlhipotesis}. Equivalently, to show the long range condition \cref{condicionlong} for $\triangle_jc$ one needs to fix paths between $t(\ee_1)$ and $o(\delta_j\ee_1)$ for each $1\leq j\leq d$, and repeat the previous computation.
\end{proof}
\subsection*{Acknowledgments}
Various parts of this paper were written during the author's stay in the Graduate School of Mathematics of Nagoya University. He is grateful to Nagoya University for the kind hospitality. He would like to thank Professor Serge Richard for its encouragement and constant support during the writing of this paper.
\printbibliography

\end{document}